\documentclass[leqno]{amsart}
\usepackage{times}
\usepackage{amsfonts,amssymb,amsmath,amsgen,amsthm}
\usepackage{hyperref}

\theoremstyle{plain}
\newtheorem{theorem}{Theorem}[section]
\newtheorem{definition}[theorem]{Definition}

\newtheorem{lemma}[theorem]{Lemma}

\newtheorem{proposition}[theorem]{Proposition}

\theoremstyle{remark}
\newtheorem{remark}[theorem]{Remark}

\def\C{{\mathbb C}}% complex numbers
\def\R{{\mathbb R}}% real numbers
\def\N{{\mathbb N}}% nonnegative integers
\def\Sch{{\mathcal S}}% Schwartz space

\def\F{\mathcal F}

\def\({\left(}
\def\){\right)}
\def\<{\left\langle}
\def\>{\right\rangle}
\def\le{\leqslant}
\def\ge{\geqslant}

\def\Tend#1#2{\mathop{\longrightarrow}\limits_{#1\rightarrow#2}}

\def\eps{\varepsilon}

\numberwithin{equation}{section}

\begin{document}
\title[Standing waves for quantum dipole]{Complementary study of the standing wave
  solutions of the
Gross-Pitaevskii equation in dipolar quantum gases}
\author[R. Carles]{R\'emi Carles}
\address{CNRS \& Univ. Montpellier~2\\Math\'ematiques
\\CC~051\\34095 Montpellier\\ France}
\email{Remi.Carles@math.cnrs.fr}
\author[H. Hajaiej]{Hichem Hajaiej}
\address{College of Sciences\\ King Saud University\\ 11451 Riyadh \\
  Department of Mathematics}
\email{hhajaiej@ksu.edu.sa}
\begin{abstract}
We study the stability of the standing wave solutions of a
Gross-Pitaevskii equation describing Bose-Einstein condensation of
dipolar quantum gases and characterize their orbit. As an
intermediate step, we consider the corresponding constrained
minimization problem and establish existence, symmetry and
uniqueness of the ground state solutions.
\end{abstract}
\thanks{This work was supported by the French ANR projects SchEq
  (ANR-12-JS01-0005-01) and \href{http://bond.math.cnrs.fr}{BoND} (ANR-13-BS01-0009-01).}
\maketitle

\section{Introduction}

Since the experimental realization of the first Bose-Einstein
condensate (BEC) by Eric Cornell and Carl Wieman in 1995, tremendous
efforts  have been undertaken by mathematicians to exploit this
achievement especially in atomic physics and optics. In the last
years, a new kind of quantum gases with dipolar interaction, which
acts between particles as a permanent magnetic or electric dipole
moment has attracted the attention of a lot of scientists. The
interactions between particles are both long-range and
non-isotropic. Describing the corresponding BEC via Gross Pitaevskii
approximation, one gets the following nonlinear Schr\"odinger
equation
\begin{equation}\label{eq:1.1}
  i\hbar \partial_t \psi = - \frac{\hbar^2}{2m} \Delta
\psi + g|\psi|^2\psi +d^2(K \ast |\psi|^2) \psi +
V(x)\psi, \quad
t \in \mathbb{R}, x \in \R^3,
\end{equation}
where $|g| = \frac{4\pi \hbar^2 N|a|}{m}$, $N \in \N$ is the number
of particles, $m$ denotes the mass of individual particles and $a$ its
corresponding scattering length. The external potential $V(x)$
describes the electromagnetic trap and has the following harmonic
confinement
\begin{equation*}
 % \label{eq:1.2}
  V(x) = \frac{|x|^2}{2}.
\end{equation*}
The factor $d^2$ denotes the strength of the dipole moment in Gaussian units
and
\begin{equation}
  \label{eq:1.3}
  K(x) = \frac{1-3\cos^2\theta}{|x|^3},
\end{equation}
where $\theta = \theta(x)$ is the angle between $x \in \mathbb{R}^3$
and the dipole axis $n \in \mathbb{R}^3$. The local term
$g|\psi|^2\psi$ describes the short-range interaction forces between
particles, while the non-local potential $K\ast 
|\psi|^2$ describes their
long-range dipolar interactions.\\
For the mathematical analysis, it is more convenient to rescale
\eqref{eq:1.1} into the following dimensionless form
\begin{equation}
  \label{eq:1.4}
 i\partial_t \psi + \frac{1}{2} \Delta \psi = \frac{|x|^2}{2} \psi
 + \lambda_1 |\psi|^2 \psi +
\lambda_2 (K \ast |\psi|^2)\psi,
\end{equation}
where
$$\lambda_1 = \frac{4\pi a N}{a_0},\quad \lambda_2 =
\frac{d^2}{\hbar w_0a^3_0},\quad
\text{and}\quad a_0 = \sqrt{\frac{\hbar}{m}}.$$
In the following, we assume that
$\lambda_1$ and $\lambda_2$ are two given real-valued parameters.\\
In \cite{CMS08}, the authors have studied the existence and uniqueness of the
equation \eqref{eq:1.4} with initial condition $\psi_0\in H^1(\R^3)$,
\begin{equation}
  \label{eq:1.5}
    i\partial_t \psi + \frac{1}{2} \Delta \psi = \frac{|x|^2}{2} \psi +
\lambda_1|\psi|^2 \psi + \lambda_2 (K \ast |\psi|^2)\psi;\quad
\psi(0,x)= \psi_0(x).
 \end{equation}
They have established that \eqref{eq:1.5} has a unique, global solution if $\lambda_1
\ge \frac{4}{3}\pi \lambda_2 \ge 0$. They called this situation
\emph{stable regime}, referring to the fact that no singularity in
formed in finite time. In this paper, we study another notion of
stability, that is, the stability of standing waves. 
They have also showed that in the \emph{unstable regime} $(\lambda_1 <
\frac{4}{3} \pi \lambda_2)$,  finite time blow up may occur,
hence the denomination.
\smallbreak

The evidence of blow-up relies on a function for which the
corresponding energy is strictly negative (\cite[Lemma 5.1]{CMS08}).
They concluded using the virial approach of Zakharov and Glassey.
Some refinements of the above result have been discussed   in
\cite[Proposition~5.4]{CMS08}.
\smallbreak

The most important  issue in view of the applications of
\eqref{eq:1.5} in atomic physics and quantum optics seems to be the
study of ground state  solutions of \eqref{eq:1.5}. These solutions
are the ``only'' observable states in experiments. A standing wave
solution of \eqref{eq:1.5} is a wave function having the particular
form $\psi(t,x) = e^{i\mu t} u(x), \mu \in \R$. Therefore
$\psi(t,x)$ is a solution of \eqref{eq:1.1} if and only if $u$
solves the following elliptic partial differential equation:
\begin{equation}
  \label{eq:1.6}
  - \frac{1}{2} \Delta u + \lambda_1 |u|^2 u + \lambda_2 (K \ast u^2) u +
\frac{|x|^2}{2} u + \mu u = 0.
\end{equation}
Ground state solutions
are the solutions of \eqref{eq:1.6} obtained by minimizing an associated
energy functional. The most common way to get such minimizers is to
consider the corresponding Weinstein functional or the constrained
energy functional.
\smallbreak

In \cite{AnSp11}, the authors have studied \eqref{eq:1.6} (without
the term $V(x))$ by using the first approach. More precisely, they
introduced the following minimization problem:
\begin{equation}
  \label{eq:1.7}
  \inf_{v \in H^1(\R^3)}J(v), \quad \text{where} \quad  J(v) =
  \frac{\|\nabla v\|^3_{L^2}\|v\|_{L^2}}
{-\lambda_1 \|v\|^4_{L^4} - \lambda_2\langle K \ast |v|^2,
|v|^2\rangle},
\end{equation}
and $\langle \cdot ,\cdot \rangle$ is the scalar product
in $L^2(\R^3)$.

Using various tricks, they were able to show that \eqref{eq:1.7} is achieved
when $\lambda_1 < \frac{4}{3} \pi \lambda_2$ if $\lambda_2 > 0$ and
$\lambda_1 < - \frac{8\pi}{3} \lambda_2$ if $\lambda_2 < 0$. They
then deduced the main result of their paper  (\cite[Theorem 1.1]{AnSp11}),
which we recall for the convenience of the reader.
\begin{theorem}[Antonelli--Sparber \cite{AnSp11}]
  Let $\lambda_1, \lambda_2 \in \R$ be such the following
condition holds
\begin{equation}\label{eq:1.8}
  \lambda_1 <
\left\{
  \begin{aligned}
    \frac{4\pi}{3} \lambda_2 &\mbox{ if } \lambda_2 > 0,\\
- \frac{8\pi}{3} \lambda_2 &\mbox{ if } \lambda_2 < 0.
  \end{aligned}
\right.
\end{equation}
Then there exists a non-negative function $u\in H^1(\R^3)$ solution
to 
\begin{equation*}
  - \frac{1}{2} \Delta u + \lambda_1 |u|^2 u + \lambda_2 (K \ast u^2) u +
 \mu u = 0,\quad \mu>0.
\end{equation*}
\end{theorem}
Note that there is no contradiction with Proposition~4.1 and
Lemma~5.1 of \cite{CMS08}, since the solitary wave constructed in
\cite{AnSp11} corresponds to an initial data with a positive energy,
while finite time blow up is established  in \cite{CMS08} only for
negative energy solutions. A complete analysis of such situations
has been done in \cite{MaCa11} and \cite{MaWa13}. The second variational formulation
associated to \eqref{eq:1.6} is the following constrained
minimization problem
\begin{equation}
  \label{eq:1.9}
  I_c = \inf \{E(u) \ ; \  u \in S_c\},
\end{equation}
where
\begin{equation}
  \label{eq:1.10}
  E(u) = \frac{1}{2} \|\nabla u\|^2_{L^2(\R^3)} + \frac{1}{2}
 \int_{\R^3} |x|^2|u|^2 +
\frac{\lambda_1}{2} \|u\|^4_{L^4(\R^3)} + \frac{\lambda_2}{2}
\int_{\R^3}(K \ast |u|^2)^2|u|^2,
\end{equation}
and
\begin{equation}
  \label{eq:1.11}
  S_c = \left\{u \in \Sigma : \int_{\R^3}
u^2 = c^2\right\},
\end{equation}
with
$$\Sigma= \left\{ u \in H^1(\R^3) :
\int_{\R^3} |x|^2 |u(x)|^2dx < \infty\right\}.$$ According to the
breakthrough paper of Grillakis, Shatah and Strauss \cite{GSS87}, the stable
solutions of \eqref{eq:1.6} are the ones obtained via the variational problem
\eqref{eq:1.9}. In \cite{AnSp11}, the authors seem to be very
skeptical concerning the
use of such approach  in this context. In \cite[p.~427]{AnSp11},
after the introduction of the energy functional they stated
``At this point it might be tempting to study \eqref{eq:1.6} via minimization
of the energy $E(u)$. However, it is
well known, that even without the dipole nonlinearity, i.e.  $\lambda_2
= 0$, this
approach fails (\dots)''. Note however that in \cite{AnSp11},
\eqref{eq:1.1} is considered in the absence of an external potential,
$V=0$. A key aspect in our approach consists in using a balance
between both nonlinear terms, the cubic one and the dipolar one. Also,
the presence of the confining potential $V$ (not necessarily
quadratic, see below) seems to be extremely
helpful in the proof, although it is not clear whether it is necessary
or not. 
\smallbreak

Their feelings have been reinforced by the approach of Bao et al. in
\cite{BCW10}, which, however, contains some flaws, which we fix in
the present paper. Note also that our method is simpler and applies
to any potential $V(|x|)$ which increases to infinity when $|x|$
tends to infinity (the radial symmetry of the potential is needed in
order to ensure that the minimzer is Steiner symmetric). Their main
result, which we revisit here, can be 
stated as follows:
\begin{theorem}\label{theo:main}
  A) If
  \begin{equation}
    \label{eq:1.12}
    \left\{
      \begin{aligned}
        \lambda_2 > 0 &\text{ and }
\lambda_1 \ge \frac{4\pi}{3} \lambda_2,\\
\mbox{ or }\\
 \lambda_2 < 0 &\text{ and } \lambda_1 \ge -
\frac{8\pi}{3} \lambda_2,
      \end{aligned}
\right.
  \end{equation}
then \eqref{eq:1.9} has a unique non-negative minimizer, which is Steiner
symmetric.\\

B) If
\begin{equation}
  \label{eq:1.13}
  \left\{
    \begin{aligned}
      \lambda_2 > 0&\text{ and } \lambda_1 < \frac{4\pi}{3} \lambda_2,\\
\mbox{ or }\\
\lambda_2 < 0 &\text{ and } \lambda_1 < - \frac{8\pi}{3} \lambda_2,
    \end{aligned}
\right.
\end{equation}
then $I_c = - \infty$.
\end{theorem}
However, the proof
of B) in \cite{BCW10} contains a flaw, which we fix here.
\smallbreak

From now on, we suppose that
$\lambda_2 > 0$. The case $\lambda_2<0$ can be treated in the same
fashion.
\smallbreak

Note that the range of $\lambda_1$ and $\lambda_2$
ensuring the existence of minimizers via Weinstein function does not
intersect at all with the one enabling us to get minimizers of
\eqref{eq:1.9}.
\smallbreak

Our paper is organized  as follows. In the next section, we fix some
notations and state some preliminary results. In
Section~\ref{sec:proof}, we
prove Theorem~\ref{theo:main}. Finally, in the last section, we 
prove the orbital stability of standing waves when \eqref{eq:1.12}
holds true. We
also characterize the orbit of standing waves.
\section{Preliminaries}
\label{sec:prelim}

\subsection{Notations}
\label{sec:notations}

The space $L^p(\R^3)$, denoted by $L^p$ for shorthand, is equipped
with the norm $|\cdot|_p$. For $w = (u,v) \in L^p \times L^p$, we set $
\|w\|^p_p =
|u|^p_p + |v|^p_p$. Similarly if
$$w = (u,v) \in H^1 \times H^1, \|w\|^2_{H^1} :=
\|w\|^2_{H^1} + \|\nabla w\|^2_{H^1,}$$
 with
$$\|\nabla w\|^2_{H^1} = |\nabla u|^2_{H^1} + |\nabla v|^2_{H^1}.$$
Recall that
$$\Sigma  = \left\{u \in H^1,\quad  |u|^2_\Sigma := |xu|^2_{L^2} +
|\nabla u|^2_2 + |u|^2_2 < \infty\right\}.$$
We set $\widetilde \Sigma=\Sigma \times \Sigma$, equipped with the
norm given by
$$\|w\|_{\widetilde \Sigma}^2= \|(u,v)\|_{\widetilde \Sigma}^2:=
|u|^2_\Sigma + |v|^2_\Sigma .$$
For $w\in \widetilde \Sigma$, we define
\begin{equation}
  \label{eq:2.1}
   \widetilde{E}(w) = \frac{1}{2}
\|\nabla w\|^2_2 + \frac{1}{2} \int_{\R^3} |x|^2 |w|^2 dx +
\frac{\lambda_1}{2} \|w\|^4_4 + \frac{\lambda_2}{2}
\int_{\R^3} (K \ast |w|^2) |w|^2 dx.
\end{equation}
Equivalently for all $c > 0$, we set
\begin{align}
  \label{eq:2.2}
  \widetilde{I}_c &= \inf \{ \widetilde{E}(w) :
 w \in  \widetilde\Sigma , \|w\|^2_2 = c^2\},\\
\widetilde{S}_c& = \{w \in \widetilde\Sigma, \|w\|^2_2 = c^2\}, \notag\\
Z_c &= \{w \in \widetilde\Sigma : \|w\|^2_2 = c^2
\mbox{ and } \widetilde{E}(w) =
 \widetilde{I}_c\}, \notag\\
W_c &= \{u \in \Sigma \cap C^1 (\R^3) \; E(u) =
I_c, |u|^2_2 = c^2 \mbox{ and }
u > 0\}. \notag
\end{align}

\subsection{Technical results}
\label{sec:tech}
We first recall two important properties of the dipole established in
\cite{CMS08}.
\begin{lemma}[Lemma~2.1 from \cite{CMS08}]\label{lem:2.4}
  The operator $\mathcal{K} : u \mapsto K \ast\; u$ can be
extended as a continuous operator on $L^p(\R^3)$ for all $1 < p < \infty$.
\end{lemma}
\begin{lemma}[Lemma~2.3 from \cite{CMS08}]\label{lem:2.5}
 Define the Fourier transform on the Schwartz space as
  \begin{equation*}
    \F u(\xi)\equiv \widehat u(\xi) =
    \int_{\R^3}e^{-ix\cdot \xi}u(x)d x,\quad u\in
    \Sch(\R^3).
  \end{equation*}
Then the Fourier transform of $K$ is given by
\begin{equation}\label{eq:2.5}
  \widehat K(\xi)= \frac{4\pi}{3} \(3\frac{\xi_3^2}{\lvert
  \xi\rvert^2}-1\)  =
  \frac{4\pi}{3} \( \frac{2\xi_3^2-\xi_1^2-\xi_2^2}{\lvert
  \xi\rvert^2}\)\in
\left[- \frac{4\pi}{3}, \frac{8\pi}{3}\right].
\end{equation}
\end{lemma}
Using Fourier transform and Plancherel's Theorem, we can rewrite the
energy functional as
\begin{equation}\label{eq:2.6}
  E(u)= \frac{1}{2}|\nabla u|^2_2 + \frac{1}{2}|xu|^2_2
+ \frac{1}{2}\int_{\R^3}\(\lambda_1+\lambda_2
\widehat{K}(\xi)\)|\widehat{\rho}(\xi)|^2 d\xi,
\end{equation}
 where $\rho(x) = |u(x)|^2$.
\smallbreak

Since $V(x)\to +\infty$ as $|x|\to\infty$, we have the standard result:
\begin{lemma}
  For all $p\in [2,6)$,  the embedding $\Sigma \hookrightarrow
  L^p(\R^3)$ is compact.
\end{lemma}
Proceeding as in \cite{HaSt04}, we have:
\begin{lemma}\label{lem:2.3}
  \begin{enumerate}
 \item The energy functional $E$ and $\widetilde{E}$ are $C^1$ on
 $\Sigma$ and $\widetilde\Sigma$, respectively.
 \item The mapping $c \mapsto I_c$ is continuous.
\end{enumerate}
\end{lemma}

\subsection{Cauchy problem}
\label{sec:cauchy}

We shall consider the initial value problem \eqref{eq:1.5} in two
situations: either $\psi$ is a scalar function, or
$\psi=(\psi_1\psi_2)$ is a vector function. In the second case,
\eqref{eq:1.5} means
\begin{equation*}
    i\partial_t \psi_j + \frac{1}{2} \Delta \psi_j = \frac{|x|^2}{2} \psi_j +
\lambda_1\(|\psi_1|^2+|\psi_2|^2\) \psi_j + \lambda_2 \(K \ast
\(|\psi_1|^2+|\psi_2|^2\)\)\psi_j,\ \  j=1,2,
\end{equation*}
along with the initial condition $\psi(0,x)= \psi_0(x)$. The main
technical remark concerning the Cauchy problem for \eqref{eq:1.5},
made in \cite{CMS08}, is that in view of Lemma~\ref{lem:2.4}, the
operator $u\mapsto (K\ast |u|^2)u$ is continuous from $L^4(\R^3)$ to
$L^{4/3}(\R^3)$. Therefore, on a technical level, it is not really
different from considering a cubic nonlinearity, for which the local
existence theory at the level of $\Sigma$ follows from Strichartz
inequalities and a fixed point argument (see
e.g. \cite{CazCourant}). Note
that because of the presence of the harmonic potentiel, working in
$H^1(\R^3)$ is not enough to ensure local well-posedness: working in
$\Sigma$ is necessary if one wants to consider a solution which
remains in $H^1(\R^3)$ (\cite{Ca08}).
Standard arguments (which can also be found
in \cite{CazCourant}) imply the conservations of mass and energy.
\begin{proposition}
  Let $\lambda_1,\lambda_2\in \R$, and $\psi_0\in\Sigma$. There exists
  $T$, depending on $\|\psi_0\|_{\Sigma}$ and a unique solution
  \begin{equation*}
    \psi\in C([-T,T];\Sigma), \quad \text{with }\psi,x\psi,\nabla
    \psi\in L^{8/3}\([-T,T];L^4(\R^3)\)
  \end{equation*}
to \eqref{eq:1.5}. The following quantities are conserved by the flow:
\begin{align*}
  &\text{Mass: }& |\psi(t)|_2=|\psi_0|_2,\quad \forall t\in [-T,T].\\
&\text{Energy: }& E\(\psi(t)\)=E(\psi_0), \quad \forall t\in [-T,T].
\end{align*}
In particular, if $\lambda_1\ge \frac{4\pi}{3}\lambda_2\ge 0$ or if
$\lambda_1\ge -\frac{8\pi}{3}\lambda_2>0$, then
$T$ can be chosen arbitrarily large, and the solution is defined for
all time. \\
If $\psi_0\in \widetilde \Sigma$, the above conclusions remain true,
up to replacing $E$ with $\widetilde E$, along with other obvious
modifications.
\end{proposition}
\subsection{Stability}
\label{sec:stab}

For a fixed $c > 0$, we use the following definition of
stability introduced by Cazenave and Lions \cite{CaLi82}.
\begin{definition}\label{def:2.1}
   The set $Z_c$ is said to be stable if $Z_c \neq
\emptyset$ and:  \\
For all $w \in Z_c$ and $\varepsilon > 0$, there exists $\delta >
0$ such that for all $\psi_0 \in  \widetilde\Sigma$, we have
$$\|\psi_0 - w\|_{\widetilde\Sigma} < \delta \Rightarrow \inf_{w \in
Z_c}\|\psi(t,x)-w\|_{\widetilde\Sigma} < \varepsilon,$$
  where
$\psi(t,x)$ is the unique solution of \eqref{eq:1.5}, corresponding to the
initial data $\psi_0$.
\end{definition}
Notice that
if $w \in Z_c$,
then there exists a Lagrange multiplier $\lambda \in \R$
such that
\begin{equation*}
  - \frac{1}{2} \Delta w + \frac{|x|^2}{2}w + \lambda_1 |w|^2 w + \lambda_2
(K \ast |w|)^2 w + \lambda w = 0.
\end{equation*}
Therefore $w =
(w_1,w_2)$ solves the following elliptic system
\begin{equation}
  \label{eq:2.4}
  \left\{
    \begin{aligned}
      - \frac{1}{2} \Delta w_1 + \frac{1}{2} |x|^2 w_1 + \lambda_1 |w|^2
w_1 + \lambda_2 (K \ast |w|^2) w_1 + \lambda w_1 &= 0,\\
- \frac{1}{2} \Delta w_2 + \frac{1}{2}|x|^2 w_2 + \lambda_1 |w|^2
w_2 + \lambda_2 (K \ast |w|^2) w_2 + \lambda w_2 &= 0.
    \end{aligned}
\right.
\end{equation}

 \section{Proof of Theorem~\ref{theo:main}}
\label{sec:proof}
 Let us first prove part A).
 Thanks to \eqref{eq:2.6}, the minimization problem \eqref{eq:1.9} can
 be rewritten in
 the following manner
 $$I_c = \inf \left\{\frac{1}{2} (|\nabla u|^2_2 + |xu|^2_2 + \int_{\R^3}
 (\lambda_1 + \lambda_2 \widehat{K}(\xi))|\widehat{\rho}(\xi)|^2d\xi ; u \in
 S_c\right\}.$$
 Now in view of \eqref{eq:2.5} and \eqref{eq:1.12}, $E(u) \ge 0$ for
 any $u \in S_c$.

 Let $\{u_n\} \subset \Sigma$ be such that $|u_n|^2_2 \to c^2$
 and $\displaystyle{\lim_{n\to \infty}}E(u_n) = I_c$. the
 above property implies that $(u_n)$ is bounded in $\Sigma$,
 therefore, we can suppose (up to a subsequence) that $u_n
 \rightharpoonup u$ in $\Sigma$. On the other hand, by the lower
 semi-continuity of the norm, we certainly have
 \begin{align}
   \label{eq:3.3}
   |xu|^2_2 + |\nabla u|^2_2 &\le \liminf |\nabla u_n|^2_2 + |x
   u_n|^2_2,\\
\label{eq:3.4}
\int_{\R^3} |u_n|^4 &\Tend n \infty \int_{\R^3} |u|^4.
 \end{align}
 Finally using Lemma~\ref{lem:2.4}, we obtain that
 \begin{equation}
   \label{eq:3.5}
   \int_{\R^3}(K \ast |u_n|^2)
 u^2_n \to \int_{\R^3}
 (K \ast |u|^2)u^2.
 \end{equation}
Relations \eqref{eq:3.3}, \eqref{eq:3.4} and \eqref{eq:3.5}
 imply that
$$E(u) \le \liminf E(u_n) \to I_c.$$
We  conclude that $E(u) = I_c$,  since $|u|_2 = c$.
\smallbreak

Now, taking into account the fact that $|\nabla|u||_2 \le |\nabla
u|_2$, we have $E(|u|) \le E(u)$ for any $u \in H^1$.
Finally, using rearrangement inequalities established by F. Brock
\cite{Br95}, we certainly get that $E(|u|^\#) \le E(|u|) \le E(u)$,
where $u^\#$ stands for the Steiner symmetrization with respect to the
$x_3$-axis.
\smallbreak

As proved in (\cite[Lemma 2.1]{BCW10}), the energy $E$ is strictly
convex, and therefore  the minimizer constructed above is unique.
\begin{remark}\label{rem:3.1}
  All minimizing sequences of \eqref{eq:1.9} are relatively
compact in $\Sigma$.
\end{remark}
Now, let us prove part B) of Theorem~\ref{theo:main}.  To reach this goal, we
need to construct an appropriate sequence of functions ensuring that
$I_c = - \infty$. In doing so, we fix a flaw in the proof of
\cite{BCW10}.
Let $f_1 \in C^\infty_0(\R^2)$ and $ f_2 \in
C^\infty_0(\R)$ be such that
\begin{equation*}
\int_{\R^3}f_1(x_1,x_2)^2 f_2(x_3)^2dx=  \(\int_{\R^2} f_1^2\)\(\int_\R f_2^2\) = c^2.
\end{equation*}
At this stage, the idea is to use anisotropy.
For $\eps,h>0$ to be made precise later, let
\begin{equation*}
  u(x) =
\frac{1}{\varepsilon} f_1\(\frac{x_1}{\varepsilon},
\frac{x_2}{\varepsilon}\) \frac{1}{\sqrt{h}}
f_2\(\frac{x_3}{h}\),\quad x\in \R^3.
\end{equation*}
Then $u \in S_c$. For $\rho=|u|^2$, we have
$$\widehat{\rho}(\xi) = \F\(|f_1|^2\)
(\varepsilon \xi_1, \varepsilon\xi_2) \F\(|f_2|^2\) (h\xi_3),\quad
\xi \in \R^3.$$
We now measure the order of magnitude, as $\varepsilon, h\to
0$ of each term in the energy, leaving out the precise value of positive
multiplicative constants. We obviously have
\begin{equation*}
  {\int}|\nabla u|^2 \approx
\frac{1}{\varepsilon^2} + \frac{1}{h^2}\quad \text{and}\quad
{\int}|x|^2|u|^2 \approx \varepsilon^2 +
h^2.
\end{equation*}
 Let $w(\xi) = \lambda_1+\lambda_2 \widehat{K}(\xi)$, $\varphi =
\left|\F\(|f_1|^2\)\right|^2$ and
$\psi = \left|\F\(|f_2|^2\)\right|^2$. Then
\begin{align*}
  \int_{\R^3}w(\xi)
|\widehat{\rho}(\xi)|^2 d\xi& = \int_{\R^3}w(\xi)
\varphi (\varepsilon \xi_1, \varepsilon \xi_2)\psi(h \xi_3)d\xi\\
&= \frac{1}{\varepsilon^2h} \int w \(\frac{\eta_1}{\varepsilon},
\frac{\eta_2}{\varepsilon},
\frac{\eta_3}{\varepsilon}\)\varphi(\eta_1,\eta_2)
\psi(\eta_3)d\eta.
\end{align*}
Now since $w$ is homogeneous of degree $0$,
\begin{equation*}
  w\(\frac{\eta_1}{\varepsilon}, \frac{\eta_2}{\varepsilon},
\frac{\eta_3}{\varepsilon}\) = \lambda_1 + \frac{4\pi}{3} \lambda_2 \;
\frac{2\varepsilon^2
\eta^2_3-h^2\eta^2_1-h^2\eta^2_2}{h^2\eta^2_1 + h^2\eta^2_2+
\varepsilon^2 \eta^2_3}.
\end{equation*}
If $h/ \eps\to +\infty$, then
$$w\(\frac{\eta_1}{\varepsilon},
\frac{\eta_2}{\varepsilon} , \frac{\eta_3}{\varepsilon}\)
\Tend {\eps,h} 0\lambda_1 - \frac{4}{3} \pi \lambda_2.$$
Now using the fact that
$\lambda_1 < \frac{4}{3} \pi \lambda_2$, $\varphi$ and $\psi$ are
non-negative functions, we certainly have that
$$\int w (\xi) |\widehat{\rho}(\xi)|^2 d\xi \approx - \frac{1}{\varepsilon^2h}
.$$
Finally, taking $h =
\sqrt{\varepsilon}$ and letting $\varepsilon$ tend to zero, we get
that $I_c = - \infty$.

\section{Stability of standing Waves}
\label{sec:stability}
In this section, we assume that
\begin{equation*}
  \lambda_1\ge \frac{4\pi}{3}\lambda_2>0.
\end{equation*}
\begin{theorem}\label{theo:stability}
The following properties hold:
\begin{itemize}
  \item[i)] For any $c > 0, I_c = \tilde{I}_c$, $Z_c \neq \emptyset$
and $Z_c$ is orbitally stable.
\item[ii)] For any $z \in Z_c, |z| \in W_c$.
\item[iii)] $Z_c = \{e^{i\theta}w, \theta \in \R\}$ where
$w$ is the unique minimize of (1.9).
\end{itemize}
\end{theorem}
\begin{proof}
   We follow the approach presented in \cite{CaLi82} and resumed in
\cite{HaSt04}. In fact to prove the stability, it suffices to show
that $Z_c \neq \emptyset$ and any minimizing sequence $\{z_n\}
\subset \widetilde{\Sigma}$ such that $\|z_n\|_2 \to c$ and
$\widetilde{E}(z_n)\to \widetilde{I}_c$ is relatively compact in
$\widetilde \Sigma$. \smallbreak

Let $z_n = (u_n, v_n) \subset \widetilde{\Sigma}$ be a sequence such
that $\|z_n\|_2 \to c$ and $\widetilde{E}(z_n)
\to \widetilde{I}_c$.\\
The first step consists in proving that $\{z_n\}$ has a subsequence
which is convergent in $\widetilde{\Sigma}$.

By the fact that $\widetilde{E}$ is a non-negative functional, we
can easily deduce that $\{z_n\}$ is bounded in $\widetilde{\Sigma}$,
therefore passing to a subsequence, one can suppose that
$$z_n \rightharpoonup z = (u,v)\quad \mbox{ in } \widetilde{\Sigma},$$
hence
$$u_n \rightharpoonup u\quad \mbox{ in } \Sigma\quad
\text{and}\quad v_n \rightharpoonup v\quad \mbox{ in }\Sigma,$$
and
\begin{equation}
  \label{eq:4.2}
  \lim_{n\to \infty} \int|\nabla u_n|^2 + |\nabla v_n|^2\quad
\mbox{ exists}.
\end{equation}
Now let $\rho_n = |z_n| = (u^2_n +
v^2_n)^{1/2}$. Clearly $\{\rho_n\} \subset \Sigma$ and for all $n \in
\mathbb{N}$ and $1 \le j \le 3$,
\begin{equation*}
\partial_j \rho_n(x) =
\left\{
\begin{aligned}
& \frac{u_n(x)\partial_j
u_n(x)+v_n(x)\partial_jv_n(x)}{\(u^2_n(x)+v^2_n(x)\)^{1/2}} & \quad
\text{if } u^2_n + v^2_n > 0, \\
&0 &\quad \text{otherwise.}
\end{aligned}
\right.
\end{equation*}
Thus
$$\widetilde{E}(z_n)-E(\rho_n) = \frac{1}{2}
\sum^3_{j=1}\int_{\{u^2_n+ v^2_n > 0\}}  \frac{\(u_n \partial_j
v_n-v_n\partial_j u_n\)^2}{u^2_n + v^2_n}dx.$$
 Therefore
$\widetilde{I}_c = \lim \widetilde{E}(z_n) \ge \limsup E(\rho_n)$.
Since
\begin{equation}
  \label{eq:4.3}
  \|z_n\|^2_2 = |\rho_n|^2_2 = c^2_n \to c^2,
\end{equation}
we get by Lemma~\ref{lem:2.3} that
$$\liminf E(\rho_n) \ge \liminf I_{c_n} \ge I_c \ge \widetilde{I}_c,$$
and hence
\begin{equation}
  \label{eq:4.4}
  \lim_{n\to \infty} E(\rho_n) =
\lim_{n\to + \infty} \widetilde{E}(z_n) = I_c =
\widetilde{I_c}.
\end{equation}
On the other hand \eqref{eq:4.2} implies that
$$\lim_{n\to \infty} \int_{\R^3}
\(|\nabla u_n|^2 + |\nabla v_n|^2 - |\nabla (u^2_n + v^2_n)^{1/2}|^2
\)dx = 0.$$
Consequently
\begin{equation}
  \label{eq:4.5}
  \lim_{n\to  \infty}
\int |\nabla u_n|^2 +  |\nabla v_n|^2 dx = \lim_{n\to
\infty} \int |\nabla(u^2_n+v^2_n)^{1/2}|dx = \lim_{n\to
\infty} \int|\nabla \rho_n|^2dx.
\end{equation}
We infer from \eqref{eq:4.3}, \eqref{eq:4.4} and
Remark~\ref{rem:3.1}  that there exists $\rho \in \Sigma$ such that
$\rho_n \to \rho$ in  $\Sigma$.

Clearly $\rho \in S_c$ and $E(\rho) = I_c$. Then $\rho \ge 0$ and
Steiner symmetric.
Moreover $\rho$ is a weak solution of \eqref{eq:1.6}.
 Thus $\rho \in
C^1(\R^3)$ and $\rho > 0$. Hence $\rho \in W_c \subset Z_c$
since $I_c = \widetilde{I_c}$. Next, we prove that $\rho = (u^2+
v^2)^{1/2}$.
 In fact $u_n \to u$ and $v_n \to v$ in
$L^2(B(0,R))$ for any $R > 0$.
 Since $[(u^2_n+v_n^2)^{1/2}
-(u^2+v^2)^{1/2}]^2 \le |u_n-u|^2 + |v_n-v|^2$, it follows that
$(u^2_n+v_n^2)^{1/2}\to (u^2+v^2)^{1/2}$ in $L^2(B(0,R))$
for any $R > 0$. Combining this and the fact that
$(u^2_n+v^2_n)^{1/2} = \rho_n \to \rho$ in $L^2$, imply that
$(u^2+v^2)^{1/2} = \rho$ a.e in $\R^3$.\\
Now to end the proof of part i) of Theorem~\ref{theo:stability}, it
suffices to prove
that $$\lim_{n\to \infty} \|\nabla z_n\|^2_2 = \|\nabla
z\|^2_2.$$ By invoking \eqref{eq:4.5}, we have that :
$$\lim_{n\to  \infty} \|\nabla z_n\|^2_2 = \lim_{n\to \infty}
|\nabla \rho_n|^2_2 = |\nabla \rho|^2_2.$$ Thus
$$\|\nabla z\|^2_2 \le \liminf \|\nabla z_n\|^2_2 \le |\nabla \rho|^2_2.$$
On the other hand, by replacing $z_n$ by $z$ in \eqref{eq:4.2}, we have that
$\|\nabla z\|^2_2 \ge |\nabla \rho|^2_2$.

This, together with the weak convergence of $z_n$ to $z$ in
$\widetilde{\Sigma}$, enables us to conclude.
\smallbreak

\noindent{\bf Proof of ii).} Let $z = (u,v) \in Z_c$ and set $\rho =
(u^2+v^2)^{1/2}$.
By the
previous proof, we know that $\rho \in W_c$ and
$$\sum^3_{j=1}\int_{\R^3} \left(\frac{u\partial_j
 v-v \partial_j u}{u^2+v^2}
\right)^2 dx = 0.$$
On the other hand, $\widetilde{E}(z) =
\widetilde{I}_c$ which implies that there exists a Lagrange
multiplier $\lambda \in \mathbb{C}$ such that :
$$E(z)\xi = \frac{\lambda}{2} \int_{\R^3} \(z \bar{\xi} +
\bar{z}\xi \)dx \quad \mbox{ for all } \xi \in \widetilde{\Sigma}.$$
Letting $\xi = z$, it follows immediately that $\lambda \in
\R$ and 
\begin{equation*}
  \left\{
    \begin{aligned}
      &- \frac{1}{2} \Delta u + \frac{|x|^2}{2} u + \lambda_1 (u^2+v^2)u
+
\lambda_2\(K \ast(u^2+v^2)\) u + \lambda u = 0,\\
&- \frac{1}{2} \Delta v + \frac{|x|^2}{2} v + \lambda_1 (u^2+v^2)v
+ \lambda_2 \(K \ast (u^2+v^2)\)v + \lambda v = 0.
    \end{aligned}
\right.
\end{equation*}
Elliptic regularity theory implies that $u,v \in
C^1(\R^3) \cap H^2(\R^3)$.
\smallbreak

Let $\Omega = \{x \in \R^3 : u(x) = 0\}$, then $\Omega$ is
closed since $u$ is continuous. Let us prove that it is also open.
Suppose that $x_0 \in \Omega$, using the fact that $v(x_0) > 0$, we
can find a Ball $B$ centered in $x_0$ such that $v(x) \neq 0$ for
any $x \in B$. Thus for $x \in B$
$$\frac{(u \partial_j v-v\partial_j u)^2}{u^2+v^2} = \(\partial_j
 \(\frac{u}{v}\)\)^2\frac{v^4}{u^2+v^2}\quad \mbox{ for } \quad 1 \le j
 \le 3.$$
This
implies that
$$\int_B \left|\nabla \(\frac{u}{v}\)\right|^2 \frac{v^4}{u^2+v^2}dx = 0.$$
Hence $\nabla (\frac{u}{v}) = 0$ on $B$.
Thus there exists a
constant $K$ such that $\frac{u}{v} = K$ on $B$. But $x_0 \in B$,
then $K \equiv 0$. We have proved that only the two alternatives
below are plausible:
\begin{itemize}
\item[a)] $u \equiv 0$ or $u \neq 0$ for all $x \in \R^3$.
\item[b)] $v \equiv 0$ or $v \neq 0$ for all $x \in \R^3$.
\end{itemize}
Now let us find the relationship between $u$ and $v$.
\smallbreak

\noindent {\bf Proof of iii).} Let $z = (w\cos \sigma,w\sin\sigma)$, $\sigma \in \R, w
\in W_c$. We denote $z$ by $z = e^{i\sigma}w$ by identifying $\C$ with
$\R^2$. Then $z \in \widetilde{S}_c$ and $\widetilde{E}(z) = E(w)
= I_c = \widetilde{I}_c$. Thus $\{e^{i\sigma}w, \sigma \in
\R, w \in W_c\} \subset Z_c$. Conversely, for $z = (u,v) \in
Z_c$, set $w = |z|$. Then $\widetilde{E}(z) = E(w) = \widetilde{I}_c
= I_c$ and $w \in W_c$. If $v \equiv 0$, $w = |w| > 0$ on
$\R^3$ and so $z = e^{i\sigma}w \in W_c$ where $\sigma = 0$
if $u > 0$ and $\sigma = \pi$ if $u < 0$ on $\R^3$.
Otherwise $v(x) \neq 0$ for all $x \in \R^3$. In this case,
it follows that $\nabla (\frac{u}{v}) = 0$ on $\R^3$.
Therefore there exists a constant $\alpha \in \R$ such that
$u = \alpha v$ on $\R^3$. Hence $z = (\alpha+i)v$ and $W =
|\alpha+i||v|$. Let $\theta \in \R$ be such that $(\alpha+i)
= |\alpha+i|e^{i\theta}$ and let $\varphi = 0$ if $v > 0$ and $\varphi = \pi$
if $v < 0$ on $\R^3$. Setting $\sigma = \theta + \varphi$,
we have  $z = (\alpha+i) v = |\alpha+i|
e^{i\theta}|v|e^{i\varphi} = we^{i\sigma}$, where $w \in W_c$.
\end{proof}

\noindent {\bf Acknowledgments}. H. Hajaiej is very grateful to Christof Sparber
and Peter Markowich for very useful discussions.

\providecommand{\bysame}{\leavevmode\hbox to3em{\hrulefill}\thinspace}
\providecommand{\MR}{\relax\ifhmode\unskip\space\fi MR }
\providecommand{\MRhref}[2]{%
  \href{http://www.ams.org/mathscinet-getitem?mr=#1}{#2}
}
\providecommand{\href}[2]{#2}

 \end{document}